\documentclass[10pt]{article}

\usepackage{amsmath}
\usepackage{amsthm}
\usepackage{amssymb}
\usepackage{latexsym}
\usepackage{pstricks}
\usepackage{graphicx, pst-plot, pst-node, pst-text, pst-tree}
\usepackage{titlefoot}
\usepackage{enumerate}
\usepackage{rotating}
\usepackage{titlesec}
\usepackage{cancel}
\usepackage[small,it]{caption}
\usepackage{setspace} 
\usepackage{color}
\definecolor{darkgray}{rgb}{0.8, 0.8, 0.8}
\definecolor{darkgray}{rgb}{0.7, 0.7, 0.7}
\definecolor{darkblue}{rgb}{0, 0, .4}

\usepackage[bookmarks]{hyperref}
\hypersetup{
        colorlinks=true,
        linkcolor=darkblue,
        anchorcolor=darkblue,
        citecolor=darkblue,
        urlcolor=darkblue,
        pdfpagemode=UseThumbs,
        pdftitle={A stack and a pop stack in series},
        pdfsubject={Combinatorics},
        pdfauthor={Rebecca Smith and Vincent Vatter},
}

\newtheorem{theorem}{Theorem}[section]
\newtheorem{proposition}[theorem]{Proposition}

\newtheoremstyle{example}{\topsep}{\topsep}%
     {}
     {}
     {\bfseries}
     {.}
     {.5em}
     {\thmname{#1}\thmnumber{ #2}}
\theoremstyle{example}

\newtheoremstyle{negexample}{\topsep}{\topsep}%
     {}
     {}
     {\bfseries}
     {.}
     {.5em}
     {\thmname{#1}\thmnumber{ #2}}
\theoremstyle{negexample}

\newcounter{todocounter}

\long\def\symbolfootnote[#1]#2{\begingroup%
\def\thefootnote{\fnsymbol{footnote}}\footnote[#1]{#2}\endgroup}

\makeatletter
\newcommand{\Rm}[1]{\expandafter\@slowromancap\romannumeral #1@}
\newfont{\footsc}{cmcsc10 at 8truept}
\newfont{\footbf}{cmbx10 at 8truept}
\newfont{\footrm}{cmr10 at 10truept}
\renewenvironment{abstract}%
                {
                  \begin{list}{}%
                     {\setlength{\rightmargin}{1in}%
                      \setlength{\leftmargin}{1in}}%
                   \item[]\ignorespaces\begin{small}}%
                 {\end{small}\unskip\end{list}}

\amssubj{05A17, 06A07}

\newpagestyle{main}[\small]{
        \headrule
        \sethead[\usepage][][]
        {\sc A Decreasing Stack and an Increasing Stack in Series}{}{\usepage}}

\title{\sc{A Decreasing Stack and an Increasing Stack in Series}}
\author{Rebecca Smith\footnotemark[\value{footnote}]\footnote{The author was partially supported by the NSA Young Investigator Grant H98230-08-1-0100.}
\\[-0.25ex]
\small Department of Mathematics\\[-1pt]
\small SUNY Brockport\\[-1pt]
\small Brockport, New York\\[15pt]}

\titleformat{\section}
        {\large\sc}
        {\thesection.}{1em}{}   

\date{}

\begin{document}
\maketitle

\pagestyle{main}

\newcommand{\s}{\mathbf{s}}
\newcommand{\m}{\mathbf{m}}
\renewcommand{\t}{\mathbf{t}}
\renewcommand{\b}{\mathbf{b}}
\newcommand{\f}{\mathbf{f}}
\newcommand{\rev}{\operatorname{rev}}
\newcommand{\dual}{\operatorname{dual}}
\newcommand{\C}{\mathcal{C}}
\newcommand{\Av}{\operatorname{Av}}

\newcommand{\inp}{\textsf{i}}
\newcommand{\tra}{\textsf{t}}
\newcommand{\out}{\textsf{o}}

%
%
%
%
%
%
%
%

\def\sdwys #1{\xHyphenate#1$\wholeString}
\def\xHyphenate#1#2\wholeString {\if#1$%
\else\say{\ensuremath{#1}}\hspace{2pt}%
\takeTheRest#2\ofTheString
\fi}
\def\takeTheRest#1\ofTheString\fi
{\fi \xHyphenate#1\wholeString}
\def\say#1{\begin{turn}{-90}\ensuremath{#1}\end{turn}}

\newenvironment{twostacks}
{
	\begin{footnotesize}
	\psset{xunit=0.0355in, yunit=0.0355in, linewidth=0.02in}
	\begin{pspicture}(0,0)(35,20)
	\psline{c-c}(0,15)(10,15)(10,2)(15,2)(15,15)(20,15)(20,2)(25,2)(25,15)(35,15)
	\rput[l](-0.5,12.5){\mbox{output}}
	\rput[r](35,12.5){\mbox{input}}
}
{
	\end{pspicture}
	\end{footnotesize}
}

\newcommand{\fillstack}[4]{%
	\rput[l](0,17.5){\ensuremath{#1}}
	\rput[c](12.6, 8.5){\begin{sideways}{\sdwys{#2}}\end{sideways}}
	\rput[c](22.6, 8.5){\begin{sideways}{\sdwys{#3}}\end{sideways}}
	\rput[r](35,17.5){\ensuremath{#4}}
}

\newcommand{\si}{%
	\psline[linecolor=darkgray]{c->}(27, 17.5)(22.5, 17.5)(22.5, 14)
}
\newcommand{\st}{%
	\psline[linecolor=darkgray]{c->}(22.5, 14)(22.5, 17.5)(12.5, 17.5)(12.5, 14)
}
\newcommand{\so}{%
	\psline[linecolor=darkgray]{c->}(12.5, 14)(12.5, 17.5)(8, 17.5)
}

\begin{abstract}
We study a sorting machine consisting of two stacks in series where the first stack has the added restriction such that entries in the stack must be in decreasing order from top to bottom.  We give the basis of the class of permutations that are sortable by this machine which shows that it is enumerated by the Schr\"oder numbers.
\end{abstract}

\section{Introduction}

A stack is a sorting device that works by a sequence of push and pop operations.  This last-in, first-out machine was shown by Knuth~\cite{knuth:the-art-of-comp:1} to sort a permutation if and only if that permutation avoids the pattern $231$.  That is, if there are not three indices $i<j<k$ with $\pi(k)<\pi(i)<\pi(j)$, then it is possible to run $\pi$ through a stack and output the identity permutation.  The class of stack-sortable permutations is enumerated by the Catalan numbers.

In the language of permutation patterns, any downset of permutations in the permutation containment ordering is a \emph{class}, and every class has a \emph{basis}, which consists of the minimal permutations not in the class.  Especially given that the basis for the class of permutations sortable by one stack contains only a single pattern of length three, considering two stacks in series is quite natural.  However, the problem becomes rather unwieldy.  In the case of two stacks in series, Murphy~\cite{murphy:restricted-perm:} showed that class of sortable permutations has an infinite basis.  The enumeration of this class also appears to be difficult.  The best known bounds are given by Albert, Atkinson, and Linton~\cite{albert:permutations-ge:}.

To get a better handle on this problem, many have considered different types of weaker sorting machines.  One such weaker machine is a stack in which the entries must increase when read from top to bottom.  Atkinson, Murphy, and Ru\v{s}kuc~\cite{atkinson:sorting-with-tw:} found an optimal result for two increasing stacks in series.  (Note that to obtain the identity permutation, the last stack will be an increasing stack even without declaring this restriction.)  Interestingly enough, this basis was still infinite but the permutation class was found to be in bijection with the permutations that avoid $1342$ as enumerated by B\'ona~\cite{bona:exact-enumerati:}.  Both enumerations were found by using a bijection with $\beta(0,1)$ trees.  

One can analogously define a \emph{decreasing stack} as a stack in which the entries must decrease when read from top to bottom.  In this paper, we study sorting with a decreasing stack followed by an increasing stack, a machine we call DI (and refer to the decreasing stack as D and the increasing stack as I).  Our main result shows that the class of DI-sortable permutations has a finite basis, $\{3142, 3241\}$.  Kremer~\cite{kremer:permutations-wi:,kremer:postscript:-per:} has shown previously that this class is enumerated by the large Schr\"oder numbers.

\section{The class of DI-sortable permutations}

We first illustrate how a permutation can be sorted with the DI sorting machine in Figure~\ref{DI_example}.  Notice that this permutation contains the pattern $2341$ and as such cannot be sorted by two increasing stacks in series.

\begin{figure}[t]
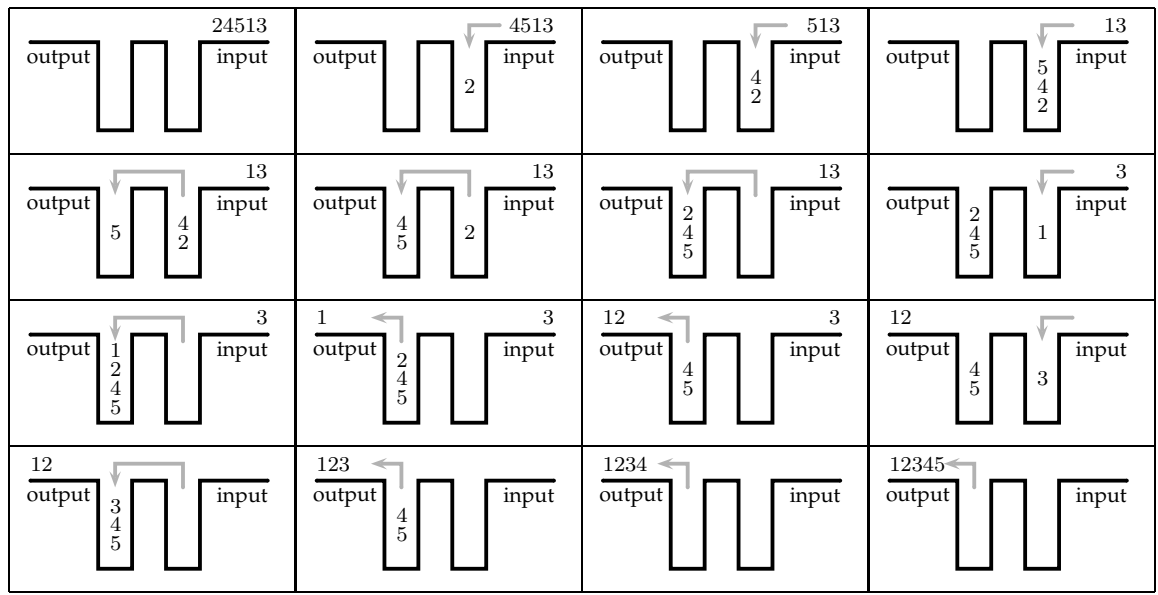

\begin{center}

\begin{tabular}{|c|c|c|c|}
\hline
\begin{twostacks}
\fillstack{}{}{}{24513}
\end{twostacks}
&
\begin{twostacks}
\fillstack{}{}{2}{4513}
\si
\end{twostacks}
&
\begin{twostacks}
\fillstack{}{}{24}{513}
\si
\end{twostacks}
&
\begin{twostacks}
\fillstack{}{}{245}{13}
\si
\end{twostacks}
\\\hline
\begin{twostacks}
\fillstack{}{5}{24}{13}
\st
\end{twostacks}
&
\begin{twostacks}
\fillstack{}{54}{2}{13}
\st
\end{twostacks}
&
\begin{twostacks}
\fillstack{}{542}{}{13}
\st
\end{twostacks}
&
\begin{twostacks}
\fillstack{}{542}{1}{3}
\si
\end{twostacks}
\\\hline
\begin{twostacks}
\fillstack{}{5421}{}{3}
\st
\end{twostacks}
&
\begin{twostacks}
\fillstack{1}{542}{}{3}
\so
\end{twostacks}
&
\begin{twostacks}
\fillstack{12}{54}{}{3}
\so
\end{twostacks}
&
\begin{twostacks}
\fillstack{12}{54}{3}{}
\si
\end{twostacks}
\\\hline
\begin{twostacks}
\fillstack{12}{543}{}{}
\st
\end{twostacks}
&
\begin{twostacks}
\fillstack{123}{54}{}{}
\so
\end{twostacks}
&
\begin{twostacks}
\fillstack{1234}{}{}{}
\so
\end{twostacks}
&
\begin{twostacks}
\fillstack{12345}{}{}{}
\so
\end{twostacks}
\\\hline
\end{tabular}

\caption{Sorting the permutation $24513$.}
\label{DI_example}
\end{center}
\end{figure}

\begin{proposition}~\label{bad_patterns}  
The permutations 3142 and 3241 are not DI-sortable.
\end{proposition}
\begin{proof}
First consider $3142$.  First one must input the $3$ into D.  As D is a decreasing stack, the $3$ must be moved to I before the $1$ enters D.  While allowed by the stack restrictions, placing the $4$ over the $1$ in D would force us to later put the $4$ above the $3$ in I, resulting in failure.  Hence, the $1$ must enter I before the $4$ enters D.  Whether the $1$ stays in I or goes to the output is of no consequence.  We still must have the $3$ in I, $4$ in D, and no good move.  The $3$ is not the next entry we wish to output.  The $4$ cannot move to I while the $3$ is in it since I is an increasing stack.  And the $2$ cannot move to D while the $4$ is in it since D is an decreasing stack.  Thus $3142$ is not sortable by DI.

Now consider $3241$.  As before, one must input the $3$ into D first and then the $3$ must be moved to I before the $2$ enters D.  Placing the $4$ over the $2$ in D would again result in having to put the $4$ above the $3$ in I.  Hence, the $2$ must enter I before the $4$ enters D.  And now we are stuck again.  The $2$ is not the next entry we want to output.  The $4$ cannot move to I while the $2$ and $3$ are in it since I is an increasing stack.  And the $1$ cannot move to D while the $4$ is in it since D is an decreasing stack.  Hence $3241$ is not sortable by DI either.
\end{proof}

\section{Algorithm}

We now give an algorithm for sorting permutations with our machine DI.    Consider the permutation that we wish to sort to be the input and on the right of the machine as is standard.

\begin{enumerate}
\item{If the top entry of the increasing stack I is the next entry of the output, then push the entry to the output.}
\item{If all of the $m$ entries in the decreasing stack D make up the next $m$ entries of the output, then push those entries to I.}
\item{Otherwise, if the next entry of the input is smaller than the top entry of I and larger than the top entry of D, then push it onto D.}
\item{Finally if neither of those moves are available, then push the entry from D to I.}
\end{enumerate}

Note that the second step of the algorithm is not actually necessary in terms of making it optimal.  However, it does get entries to the output more quickly.

Using our algorithm, we now show that the set of forbidden patterns $\{3142,3241\}$ classify our sortable permutations. 

\begin{proposition}  The basis for the class of permutation sortable by DI using the given algorithm is $\{3142, 3241\}$.
\end{proposition}
\begin{proof}
Suppose $\pi$ is a permutation not sortable by DI.  Consider the point at which the algorithm fails.  Then there must be an entry $c$ atop the increasing stack I that is not the next entry to be output.  The decreasing stack D will be empty as those entries move from D to I as part of the algorithm. Also, the first of the remaining input entries, say $d$, must be larger than the entry $c$.  Finally, the smallest entry $b$ that we were not able to successfully output must appear later in the input since any smaller entries in D can always move to I and then be output at their turn.  Clearly, the entry $c$ appears before $d$ in the permutation $\pi$.

At the point when $c$ moved from D to I, we know that every entry in I must have been larger than $c$.  Also, moving $d$ atop $c$ (or any other entry smaller than $d$) in D is  preferred by the algorithm if $d$ is also smaller than the top entry of I, so there must have been a situation that forced $c$ to be moved to I before $d$ could enter.  

At this point of moving when $c$ moved from D to I, it could have been the case that the top entry $h$ of D was smaller than the first available entry of the input, but larger than $c$.   Note that this first entry of the input can be assumed to be $d$, because moving $c$ to the output does not allow for this large entry to enter D either.   Also note that if $h$ appeared after $c$ in $\pi$, there would have been another entry $c<h<h_2<d$ that forced $h$ from D to I.  Repeat this process until we find one that did appear before $c$ in $\pi$.   Hence we have the pattern $hcdb$, that is $3241$ in $\pi$.

The other possibility is that there could have been an entry $a<c$ between $c$ and $d$ in $\pi$ that could not enter D until $c$ was pushed to I.  In this case, this entry $a$ must later move to I itself, and since $c$ is atop I at the time $d$ is at the front of the input, $a$ must have been output.  Hence we have the pattern $cadb$, that is $3142$ in $\pi$.
\end{proof}

\begin{proposition}
The given algorithm is optimal for sorting permutations with the DI machine.
\end{proposition}
\begin{proof}
Let $\pi$ be a permutation that is sortable by DI, but not when using this algorithm.  Then $\pi$ contains either a $3142$ or a $3241$ pattern, both which we have already shown to not be DI-sortable in Proposition~\ref{bad_patterns}.
\end{proof}

The previous two propositions give us the basis for the DI-sortable permutations.

\begin{theorem}  The basis for the class of permutations sortable by DI  is $\{3142, 3241\}$.
\end{theorem}

\section{Concluding Remarks}~\label{concluding}

The DI-sortable permutations have previously been enumerated:

\begin{theorem}[Kremer~\cite{kremer:permutations-wi:,kremer:postscript:-per:}]
The number of DI-sortable permutations of length $n$ is equal to the $n-1^{\rm st}$ large Schr\"oder number.
\end{theorem}

Thus it would be of interest if there is a natural bijection between the actions taken by the machine when sorting DI-sortable permutations and one of the sets of combinatorial objects well-known to be counted by the large Schr\"oder numbers, such as Schr\"oder paths.

One possible approach would be to use the fact that the small Schr\"oder numbers are (starting at $n=1$) half of the large Schr\"oder numbers.  Also, the Schr\"oder paths that have a diagonal step on the main diagonal are in bijection with the small Schr\"oder numbers.  Similarly, the sum decomposable DI-sortable permutations are counted by the small Schr\"oder numbers.  (This can be shown using the fact that this class is closed under the direct sum operation.)  Could diagonal steps on the main diagonal be in some way correlated with points at which both of the stacks are empty during the sorting process?

Another problem would be to see what happens if more than one decreasing stack is put in series with an increasing stack, the DDI machine, or perhaps even the DIDI machine.  While multiple stacks in series are difficult to contend with, the decreasing restriction might make the problem more manageable.

\bigskip
\noindent{\bf Acknowledgments:} The author would like to thank Daniel Rose for his \LaTeX\  macros for drawing two stacks in series.

\bibliographystyle{acm}
\bibliography{../refs}

\def\cprime{$'$}
\begin{thebibliography}{1}

\bibitem{albert:permutations-ge:}
{\sc Albert, M.~H., Atkinson, M., and Linton, S.}
\newblock Permutations generated by stacks and deques.
\newblock {\em Ann. Comb. 14}, 1 (2010), 3--16.

\bibitem{atkinson:sorting-with-tw:}
{\sc Atkinson, M.~D., Murphy, M.~M., and Ru{\v{s}}kuc, N.}
\newblock Sorting with two ordered stacks in series.
\newblock {\em Theoret. Comput. Sci. 289}, 1 (2002), 205--223.

\bibitem{bona:exact-enumerati:}
{\sc B{\'o}na, M.}
\newblock Exact enumeration of {$1342$}-avoiding permutations: a close link
  with labeled trees and planar maps.
\newblock {\em J. Combin. Theory Ser. A 80}, 2 (1997), 257--272.

\bibitem{knuth:the-art-of-comp:1}
{\sc Knuth, D.~E.}
\newblock {\em The art of computer programming. {V}olume 1}.
\newblock Addison-Wesley Publishing Co., Reading, Mass., 1968.
\newblock Fundamental Algorithms.

\bibitem{kremer:permutations-wi:}
{\sc Kremer, D.}
\newblock Permutations with forbidden subsequences and a generalized
  {S}chr\"oder number.
\newblock {\em Discrete Math. 218}, 1-3 (2000), 121--130.

\bibitem{kremer:postscript:-per:}
{\sc Kremer, D.}
\newblock Postscript: ``{P}ermutations with forbidden subsequences and a
  generalized {S}chr\"oder number''.
\newblock {\em Discrete Math. 270}, 1-3 (2003), 333--334.

\bibitem{murphy:restricted-perm:}
{\sc Murphy, M.~M.}
\newblock {\em Restricted Permutations, Antichains, Atomic Classes, and Stack
  Sorting}.
\newblock PhD thesis, Univ. of St Andrews, 2002.

\end{thebibliography}

\end{document}